\documentclass[11pt]{article}
\pdfoutput=1

\usepackage[margin=1.3in]{geometry}

\usepackage{graphicx}
\usepackage{amssymb,amsmath}
\allowdisplaybreaks[4]
\usepackage{amsthm,enumerate}
\usepackage{hyperref,xcolor}
\usepackage{mathrsfs}
\usepackage{extarrows}
\usepackage{textcomp}

\makeatletter

\newdimen\bibspace
\makeatother

\numberwithin{equation}{section}

\newtheorem{theorem}{Theorem}[section]
\newtheorem{lemma}[theorem]{Lemma}

\newtheorem{corollary}[theorem]{Corollary}
\newtheorem{remark}[theorem]{Remark}

\def\<{\langle}
\def\>{\rangle}

\def\hat{\widehat}

\def\XXint#1#2#3{{\setbox0=\hbox{$#1{#2#3}{\int}$ }
\vcenter{\hbox{$#2#3$ }}\kern-.6\wd0}}

\def\hat{\widehat}

\begin{document}

\title{Asymptotic behavior of solutions to the Monge--Amp\`ere equations with slow convergence rate at infinity}

\author{Zixiao Liu and Jiguang Bao\footnote{
 Supported by the National Key Research and Development Program of China (No. 2020YFA0712904) and
 National Natural Science Foundation of China (No. 11871102 and No. 11631002).}}
\date{\today}

\maketitle

\begin{abstract}
We consider the asymptotic behavior of solutions to the Monge--Amp\`ere equations with slow convergence rate at infinity and fulfill previous results under faster convergence rate by Bao--Li--Zhang [Calc. Var PDE. 52(2015). pp. 39-63]. Different from known results, we  obtain the limit of Hessian and/or gradient of solution at infinity relying on the  convergence rate.  The basic idea is to use a revised level set method, the spherical harmonic expansion and the iteration method.
 \\[1mm]
\end{abstract}

\section{Introduction}

We consider convex viscosity solutions to the Monge--Amp\`ere equations
\begin{equation}\label{equ:MA}
\det\left(D^{2} u\right)=f(x) \quad\text{in } \mathbb{R}^{n},
\end{equation}
where $D^2u$ denotes the Hessian matrix of $u$ and $f\in C^m(\mathbb R^n)$ satisfies

\begin{equation}\label{equ:cond-f}
\limsup _{|x| \rightarrow \infty}|x|^{\zeta+k}\left|D^{k}(f(x)-f(\infty))\right|<\infty, \quad \forall~ k=0,1,2, \cdots, m
\end{equation}
for some $f(\infty)>0$,  $\zeta>0$ and $m\geq 2$.

Equation \eqref{equ:MA} with $f$ being a constant origins from 2-dimensional minimal surfaces \cite{Jorgens}, improper affine geometry \cite{Calabi,Pogorelov} etc.
The importance of $f$ not being identical to a constant is mentioned in Calabi \cite{Calabi}, Trudinger--Wang \cite{Trudinger-Wang-MA-GeomeAppl} etc. As pointed out in \cite{Nirenberg-WeylnMinkovskiProb,Bao-Li-Zhang-ExteriorBerns-MA,Chong-Rongli-Bao-SecondBoundary-SPL,Liu-Bao-2021-Expansion-LagMeanC}, such equations are also related to the Weyl and Minkovski problems, the Plateau type problems, affine geometry and the mean curvature equations of gradient graphs in weighted space etc.

When $f(x)\equiv f(\infty)>0$, the theorem by J\"orgens \cite{Jorgens} ($n=2$), Calabi \cite{Calabi} ($n\leq 5$) and Pogorelov \cite{Pogorelov} ($n\geq 2$) states that any classical convex solution of \eqref{equ:MA} must be a quadratic polynomial. For $n=2$, a classical solution of \eqref{equ:MA} is either convex or concave and thus the result hold without the convexity assumptions. For different proofs and extensions, we refer to Cheng--Yau \cite{ChengandYau}, Caffarelli \cite{Caffarelli-InteriorEstimates-MA}, Jost--Xin \cite{JostandXin}, Fu \cite{Fu-Bernstein-SPL}, Li--Xu--Simon--Jia \cite{Book-Li-Xu-Simon-Jia-MA}, Warren \cite{Warren-Calibrations-MA} etc.

When $f(x)-f(\infty)$ have compact support, Caffarelli--Li \cite{Caffarelli-Li-ExtensionJCP} proved that any convex viscosity solution must be asymptotic to quadratic polynomial at infinity (with additional $\ln$-term when $n=2$). Such asymptotic behavior has been refined further with an expansion of error at infinity by Hong \cite{Hong-Remark-MA} (for $|x|^{2-n}$ order with $n\geq 3$) and Liu--Bao \cite{Liu-Bao-2020-ExpansionSPL,Liu-Bao-2021-Dim2} (for higher order with $n=2$ and $n\geq 3$).

When $f(x)-f(\infty)$ vanish at infinity, Bao--Li--Zhang \cite{Bao-Li-Zhang-ExteriorBerns-MA} proved the following asymptotic behavior result, which is an extension to previous results by J\"orgens \cite{Jorgens}--Calabi \cite{Calabi}--Pogorelov \cite{Pogorelov} and Caffarelli--Li \cite{Caffarelli-Li-ExtensionJCP}.

Hereinafter, we let $\mathtt{Sym}(n)$ denote the set of symmetric $n\times n$ matrix,  $x^T$ denote the transpose of vector $x\in\mathbb R^n$ and convex viscosity solutions are defined as in \cite{Book-Caffarelli-Cabre-FullyNonlinear,Caffarelli-Li-ExtensionJCP}. We will write $\varphi(x)=O_m(|x|^{-k_1}(\ln|x|)^{k_2})$ with $m\in\mathbb N, k_1, k_2\geq 0$ if $\varphi$ satisfies
\[
|D^k\varphi(x)|=O(|x|^{-k_1-k}(\ln|x|)^{k_2})
\quad\text{as}~|x|\rightarrow+\infty
\] for all $0\leq k\leq m$.
\begin{theorem}[Bao--Li--Zhang\cite{Bao-Li-Zhang-ExteriorBerns-MA}]
  Let $u\in C^0(\mathbb R^n)$ be a convex viscosity solution of \eqref{equ:MA} where $f\in C^m(\mathbb R^n)$ satisfies \eqref{equ:cond-f} with $\zeta>2$ and $m\geq 3$. If $n\geq 3$, there exist $0<A\in\mathtt{Sym}(n)$ satisfying $\det A=f(\infty)$, $b\in\mathbb R^n$ and $c\in\mathbb R$ such that
  \begin{equation}\label{equ:asym-ori-n>3}
  u(x)-\left(\frac{1}{2}x^TAx+ b \cdot x+c\right)=\left\{
  \begin{array}{lll}
  O_{m+1}(|x|^{2-\min\{\zeta,n\}}), & \text{if }\zeta\not=n,\\
  O_{m+1}(|x|^{2-n}(\ln|x|)), & \text{if }\zeta=n,\\
  \end{array}
  \right.
  \end{equation}
  as $|x|\rightarrow\infty$. If $n=2$, there exist $A, b ,c$ as above and $d=\frac{1}{4\pi}\int_{\mathbb R^2}(f(x)-1)\mathrm dx$ such that
  \begin{equation}\label{equ:asym-ori-n=2}
  u(x)-\left(\frac{1}{2}x^TAx+ b \cdot x+c+d\ln(x^TAx)\right)=
  O_{m+1}(|x|^{2-\bar\zeta})
  \end{equation}
  as $|x|\rightarrow\infty$ for any $\bar\zeta<\min\{\zeta,3\}$.
\end{theorem}

\begin{remark}\label{rem:lostCases}
  As discussed in Theorem 1.1 of \cite{Liu-Bao-2021-Expansion-LagMeanC}, the original statement in Bao--Li--Zhang \cite{Bao-Li-Zhang-ExteriorBerns-MA}  dropped the possibility that when $\zeta=n$ in $n\geq 3$ case, which leads to the difference between \eqref{equ:asym-ori-n>3} and (1.2) in \cite{Bao-Li-Zhang-ExteriorBerns-MA}.
  Furthermore, from (2.33) in \cite{Bao-Li-Zhang-ExteriorBerns-MA}, it seems that they also dropped the term of $O(|x|^{-1})$ order in spherical harmonic expansion at infinity in \eqref{equ:asym-ori-n=2}, which makes the range of $\bar\zeta$ different from the original statement in \cite{Bao-Li-Zhang-ExteriorBerns-MA}. See also Theorem 1.1 of \cite{Liu-Bao-2021-Dim2-MeanCur}.
\end{remark}

In previous work by the authors \cite{Liu-Bao-2021-Expansion-LagMeanC,Liu-Bao-2021-Dim2-MeanCur}, when $n\geq 3$ the requirement  $m\geq 3$ is reduced into $m\geq 2$ and when $n=2$ the asymptotic behavior \eqref{equ:asym-ori-n=2} is further refined into
\begin{equation}\label{equ-optimality}
\begin{array}{llllll}
&\displaystyle  u(x)-\left(\frac{1}{2}x^TAx+ b \cdot x+c+d\ln(x^TAx)\right)\\
=&\displaystyle
  \left\{
  \begin{array}{lllll}
O_{m+1}\left(|x|^{2-\min\{\zeta,3\}}\right), & \text{if }\zeta\not=3,\\
O_{m+1}\left(|x|^{-1}(\ln|x|)\right), & \text{if }\zeta=3,\\
\end{array}
  \right.\\
\end{array}
  \end{equation}
as $|x|\rightarrow\infty$. Furthermore, higher order asymptotic expansions when $\zeta$ is larger are also given in \cite{Liu-Bao-2021-Expansion-LagMeanC,Liu-Bao-2021-Dim2-MeanCur}.

As pointed out by Bao--Li--Zhang \cite{Bao-Li-Zhang-ExteriorBerns-MA}, by considering radially symmetric solutions, $\zeta>2$ is optimal such that $u$ converge to a quadratic function (with additional $\ln$-term when $n=2$) at infinity. See also the example in Section \ref{seclabel-Radial} below.

We consider under slow convergence speed $0<\zeta\leq 2$ and prove
the asymptotic behavior at infinity. The statement is separated into two parts since the requirement on the regularity of $f$ when $n\geq 3$ is different from $n=2$ case.

\begin{theorem}\label{thm:main-n>3}
  Let $u\in C^0(\mathbb R^n)$ be a convex viscosity solution of \eqref{equ:MA} where $n\geq 3$ and $f\in C^m(\mathbb R^n)$ satisfies \eqref{equ:cond-f} for some $0<\zeta\leq 2$ and $m\geq 2$. Then there exist $0<A\in\mathtt{Sym}(n)$ satisfying $\det A=f(\infty)$ and $ b \in\mathbb R^n$ such that
  \begin{equation}\label{equ-beha-n>3}
  u(x)-\frac{1}{2}x^TAx=\left\{
  \begin{array}{lllll}
    O_{m+1}(|x|^{2-\zeta}), & \text{if }0<\zeta<1,\\
    O_{m+1}(|x|(\ln|x|)), & \text{if }\zeta=1,\\
     b \cdot x+O_{m+1}(|x|^{2-\zeta}), & \text{if }1<\zeta<2,\\
     b \cdot x+O_{m+1}(\ln|x|), & \text{if }\zeta=2,\\
  \end{array}
  \right.
  \end{equation}
  as $|x|\rightarrow\infty$.
\end{theorem}
\begin{theorem}\label{thm:main-n=2}
  Let $u\in C^0(\mathbb R^2)$ be a convex viscosity solution of \eqref{equ:MA} where $f\in C^m(\mathbb R^2)$ satisfies \eqref{equ:cond-f} for some $0<\zeta\leq 2$ and $m\geq 3$. Then there exist $0<A\in\mathtt{Sym}(2)$ satisfying $\det A=f(\infty)$ and $ b \in\mathbb R^2$ such that
  \begin{equation}\label{equ-beha-n=2}
  u(x)-\frac{1}{2}x^TAx=\left\{
  \begin{array}{lllll}
    O_{m+1}(|x|^{2-\zeta}), & \text{if }0<\zeta<1,\\
    O_{m+1}(|x|(\ln|x|)), & \text{if }\zeta=1,\\
     b \cdot x+O_{m+1}(|x|^{2-\zeta}), & \text{if }1<\zeta<2,\\
     b \cdot x+O_{m+1}((\ln|x|)^2), & \text{if }\zeta=2,\\
  \end{array}
  \right.
  \end{equation}
  as $|x|\rightarrow\infty$.
\end{theorem}

\begin{remark}
  For $\zeta\not=1$, the optimality of \eqref{equ-beha-n>3} and \eqref{equ-beha-n=2} can be verified by radially symmetric solutions, where  $f(x)$ and $u(x)$ are as in \eqref{equ:example-0} and  \eqref{equ:example} below.
  For $\zeta=1$, whether \eqref{equ-beha-n>3} and \eqref{equ-beha-n=2} are optimal remains a problem for now.
\end{remark}

The paper is organized as follows. In section \ref{seclabel-Radial}
we give the asymptotic expansion of radially symmetric solutions where $f=1+|x|^{-\zeta}$ with $0<\zeta\leq 2$ at infinity.
In section \ref{seclabel-QuadraticTerm} we capture the quadratic term of $u$ given in Theorems \ref{thm:main-n>3} and \ref{thm:main-n=2} at infinity i.e., there exist $0<A\in\mathtt{Sym}(n)$ satisfying $\det A=f(\infty)$ and $C,\epsilon>0$ such that
\begin{equation}\label{equ-roughConverge}
\left|u(x)-\frac{1}{2}x^TAx\right|\leq C|x|^{2-\epsilon}.
\end{equation}
In section \ref{seclabel-Poissonequ} we prepare some necessary results on existence of solution to Poisson equations on exterior domain.  In sections
\ref{seclabel-n>3} and \ref{seclabel-n=2} we prove Theorems \ref{thm:main-n>3} and \ref{thm:main-n=2} respectively.

\section{Radially symmetric examples}\label{seclabel-Radial}

Consider strictly positive radially symmetric function $f\in C^{\infty}(\mathbb R^n)$ with

  \begin{equation}\label{equ:example-0}
  f(x)=\left\{
  \begin{array}{llll}
    1, & 0\leq |x|\leq 1,\\
    1+|x|^{-\zeta}, & |x|>2,\\
  \end{array}
  \right.
  \end{equation}
  where $0<\zeta\leq 2$.
  By a direct computation, we have the following radially symmetric solution of \eqref{equ:MA}
\begin{equation}\label{equ:example}
u(x)=n^{\frac{1}{n}} \int_{0}^{r}\left(\int_{0}^{s} t^{n-1} f(t)  \mathrm{d} t\right)^{\frac{1}{n}}  \mathrm{d} s,
\end{equation}
where $r:=|x|\geq 0$. We shall obtain asymptotic expansion at infinity.

\begin{theorem}\label{thm:radial}
  Let $f(x),u(x)$ be as in \eqref{equ:example-0} and \eqref{equ:example}. Then for sufficiently large $|x|$ we have the following asymptotic expansion at infinity.
When $0<\zeta<n$,
\begin{equation}\label{equ:radial}
\begin{array}{llll}
u(x)=&\displaystyle \frac{r^2}{2}+C_1\ln r+C_2\\
&\displaystyle+\sum_{j=1}^{\infty}\sum_{k=0,\cdots,j\atop
\zeta k+n(j-k)\neq 2}
 \dfrac{(\frac{1}{n})\cdots (\frac{1}{n}-j+1)}{k!(j-k)!(2-\zeta k-n(j-k))}\frac{n^jC_0^{j-k}}{(n-\zeta)^k}r^{2-\zeta k-n(j-k)}.\\
\end{array}
\end{equation}
When $\zeta=n$,
\begin{equation}\label{equ:radial-22}
\begin{array}{llll}
u(x)=&\displaystyle \frac{r^2}{2}+\frac{1}{2}(\ln r)^2+C_3\ln r+C_4\\
&\displaystyle -\sum_{j=2}^{\infty}\sum_{k=0}^j\sum_{l=0}^{k}
\dfrac{(\frac{1}{2})\cdots (\frac{1}{2}-j+1)\cdot 2^{j-l-1}C_3^{j-k}}{(j-k)!(k-l)!(j-1)^{l+1}}
 r^{2-2j}(\ln r)^{k-l}.
 \end{array}
\end{equation}
Here the  constants $C_0,C_1,C_2,C_3, C_4$ are given below in the proof.

\end{theorem}
\begin{proof}
  When $\zeta<n$,   for all $r>R>2$, we have from \eqref{equ:example} that
\[\begin{array}{llll}
    u(x)&=&\displaystyle  C_R+n^{\frac{1}{n}}\int_R^r\left(\int_{0}^{s} t^{n-1} f(t) \mathrm{d} t\right)^{\frac{1}{n}} \mathrm{~d} s\\
  &=&\displaystyle C_R+n^{\frac{1}{n}}\int_R^r\left(
  \frac{s^n}{n}+\frac{s^{n-\zeta}}{n-\zeta}+C_0
  \right)^{\frac{1}{n}} \mathrm{~d} s\\
  &=&\displaystyle C_R+\int_R^rs\left(1+\left(\frac{n}{n-\zeta}s^{-\zeta}+nC_0s^{-n}\right)\right)^{\frac{1}{n}}\mathrm ds,\\
\end{array}
\]
 where
\[
  C_R:=n^{\frac{1}{n}}\int_0^R\left(\int_{0}^{s} t^{n-1} f(t) \mathrm{d} t\right)^{\frac{1}{n}} \mathrm{~d} s\quad\text{and}\quad
  C_0=\int_0^2t^{n-1}f(t)\mathrm dt-\frac{2^n}{n}-\frac{2^{n-\zeta}}{n-\zeta}.
\]
Choose $R=R(n,\zeta,C_0)>2$  such that $\frac{n}{n-\zeta}R^{-\zeta}+n|C_0|R^{-n}<1$ and for all $r>R$ we have
\[\begin{array}{lllll}
 u(x) &  =&\displaystyle
  C_R+\int_R^rs\left(
  1+\sum_{j=1}^{\infty}\dfrac{(\frac{1}{n})\cdots (\frac{1}{n}-j+1)}{j!}\left(\frac{n}{n-\zeta}s^{-\zeta}+nC_0s^{-n}\right)^j
  \right)\mathrm ds\\
  &=&\displaystyle
  C_R+\int_R^rs\left(
  1+\sum_{j=1}^{\infty}\sum_{k=0}^j\dfrac{(\frac{1}{n})\cdots (\frac{1}{n}-j+1)}{k!(j-k)!}
  \left(\frac{n}{n-\zeta}s^{-\zeta}\right)^{k}(nC_0s^{-n})^{j-k}
  \right)\mathrm ds\\
  &=&\displaystyle
  \frac{r^2}{2}+C_R-\frac{R^2}{2}+
  \sum_{j=1}^{\infty}\sum_{k=0}^j
  \dfrac{(\frac{1}{n})\cdots (\frac{1}{n}-j+1)}{k!(j-k)!}\frac{n^jC_0^{j-k}}{(n-\zeta)^k}
  \int_R^rs^{1-\zeta k-n(j-k)}\mathrm ds.
\end{array}
\]

By a direct computation,  we obtain the desired result \eqref{equ:radial} with
\[
C_1
=\sum_{j=1}^{\infty}\sum_{k=0,\cdots,j\atop \zeta k+n(j-k)=2}\frac{\left(\frac{1}{n}\right) \cdots\left(\frac{1}{n}-j+1\right)}{k !(j-k) !} \frac{n^{j} C_{0}^{j-k}}{(n-\zeta)^{k}}
\]
and
\[
C_2=
 C_R-\frac{R^2}{2}-C_1\ln R
 -\sum_{j=1}^{\infty}\sum_{k=0,\cdots,j\atop \zeta k+n(j-k)\neq 2}
 \dfrac{(\frac{1}{n})\cdots (\frac{1}{n}-j+1)}{k!(j-k)!(2-\zeta k-n(j-k))}\frac{n^jC_0^{j-k}}{(n-\zeta)^k}R^{2-\zeta k-n(j-k)}.
\]

  When $\zeta=n$, since $0<\zeta\leq 2$ and $n\geq 2$,  the only possibility is $\zeta=n=2$. Thus for all $r>R>2$, we have from \eqref{equ:example} that
  \[
  \begin{array}{lllll}
    u(x)&=&\displaystyle  C_R+2^{\frac{1}{2}}\int_R^r\left(\int_{0}^{s} t f(t) \mathrm{d} t\right)^{\frac{1}{2}} \mathrm{~d} s\\
  &=&\displaystyle C_R+2^{\frac{1}{2}}\int_R^r\left(
  \frac{s^2}{2}+\ln s+C_3
  \right)^{\frac{1}{2}} \mathrm{~d} s\\
  &=&\displaystyle C_R+\int_R^rs\left(1+\left(2s^{-2}\ln s+2C_3s^{-2}\right)\right)^{\frac{1}{2}}\mathrm ds,\\
\end{array}
\]
where $C_R$ is as above and
\[
  C_3:=\int_0^2tf(t)\mathrm dt-2-\ln 2.
\]
Choose  $R=R(n,\zeta,C_3)>2$  such that $2 R^{-2} \ln R+2 |C_{3}| R^{-2}<1$ and
for all $r>R$ we have
\[\begin{array}{lllll}
  u(x)&=&\displaystyle
  C_R+\int_R^rs\left(
  1+\sum_{j=1}^{\infty}\dfrac{(\frac{1}{2})\cdots (\frac{1}{2}-j+1)}{j!}\left(2 s^{-2} \ln s+2 C_{3} s^{-2}\right)^j
  \right)\mathrm ds\\
  &=&\displaystyle
  C_R+\int_R^rs\left(
  1+\sum_{j=1}^{\infty}\sum_{k=0}^j\dfrac{(\frac{1}{2})\cdots (\frac{1}{2}-j+1)}{k!(j-k)!}
  \left(2 s^{-2} \ln s\right)^{k}(2C_3s^{-2})^{j-k}
  \right)\mathrm ds\\
  &=&\displaystyle
  \frac{r^2}{2}+C_R-\frac{R^2}{2}+
  \sum_{j=1}^{\infty}\sum_{k=0}^j
  \dfrac{(\frac{1}{2})\cdots (\frac{1}{2}-j+1)}{k!(j-k)!}\cdot 2^j C_3^{j-k}
  \int_R^rs^{1-2j}(\ln s)^k\mathrm ds.
  \end{array}
  \]

By a direct computation,
\[
\int_{R}^rs^{-1}\ln s\mathrm ds=\frac{1}{2}\left((\ln r)^2-(\ln R)^2\right)
\]
and for $j=2,3,\cdots$,
\[
\begin{array}{lllll}
&\displaystyle
\int_R^rs^{1-2j}(\ln s)^k\mathrm ds\\
=&\displaystyle \dfrac{1}{2-2j}\left(
r^{2-2j}(\ln r)^k-R^{2-2j}(\ln R)^k-k\int_R^rs^{1-2j}(\ln s)^{k-1}\mathrm ds
\right)\\
=&\displaystyle \dfrac{(r^{2-2j}(\ln r)^k-R^{2-2j}(\ln R)^k)}{2-2j}\\
&\displaystyle -\dfrac{k}{(2-2j)^2}\left(r^{2-2j}(\ln r)^{k-1}-R^{2-2j}(\ln R)^{k-1}
-(k-1)\int_R^rs^{1-2j}(\ln s)^{k-2}\mathrm ds
\right)\\
=&\displaystyle \dfrac{(r^{2-2j}(\ln r)^k-R^{2-2j}(\ln R)^k)}{2-2j}-
\dfrac{k(r^{2-2j}(\ln r)^{k-1}-R^{2-2j}(\ln R)^{k-1})}{(2-2j)^2}\\
&\displaystyle +\dfrac{k(k-1)}{(2-2j)^3}
\left(r^{2-2j}(\ln r)^{k-2}-R^{2-2j}(\ln R)^{k-2}
-(k-2)\int_R^rs^{1-2j}(\ln s)^{k-3}\mathrm ds
\right)\\
=&\cdots\\
=&\displaystyle -\sum_{l=0}^{k}\dfrac{k!}{(2j-2)^{l+1}(k-l)!}\left(
r^{2-2j}(\ln r)^{k-l}-R^{2-2j}(\ln R)^{k-l}
\right).\\
\end{array}
\]
Consequently, we obtain the desired result \eqref{equ:radial-22} with
\[
\begin{array}{lll}
C_4&=& C_R-\frac{R^2}{2}-\frac{1}{2}(\ln R)^2-C_3\ln R\\
&&\displaystyle+\sum_{j=2}^{\infty}\sum_{k=0}^j\sum_{l=0}^{k}
\dfrac{(\frac{1}{2})\cdots (\frac{1}{2}-j+1)\cdot 2^{j-l-1}C_3^{j-k}}{(j-k)!(k-l)!(j-1)^{l+1}} \cdot
 R^{2-2j}(\ln R)^{k-l}.\\
\end{array}
\]
\end{proof}

By the asymptotic expansion results in Theorem \ref{thm:radial}, we have the following corollary which proves Theorems \ref{thm:main-n>3} and \ref{thm:main-n=2} for radially symmetric cases, and show the optimality of \eqref{equ-beha-n>3} and \eqref{equ-beha-n=2} for $\zeta\neq 1$.

\begin{corollary}

Let $f(x), u(x)$ be as in \eqref{equ:example-0} and \eqref{equ:example}.
When $n\geq 3$, we have
\begin{equation}\label{equ:example-result-1}
u(x)=\dfrac{1}{2}|x|^2+\left\{
\begin{array}{llll}
  O(|x|^{2-\zeta}), & \text{if }\zeta<2,\\
  O(\ln|x|), & \text{if }\zeta=2,\\
\end{array}
\right.
\end{equation}
as $|x|\rightarrow\infty$. When $n=2$, we have
\begin{equation}\label{equ:example-result-2}
u(x)=\dfrac{1}{2}|x|^2+\left\{
\begin{array}{llll}
  O(|x|^{2-\zeta}), & \text{if }\zeta<2,\\
  O((\ln|x|)^2), & \text{if }\zeta=2,\\
\end{array}
\right.
\end{equation}
as $|x|\rightarrow\infty$. The estimates above are also optimal.
\end{corollary}
\begin{proof}
  When $\zeta<n$, we have that
\[
\zeta k+n(j-k)=2\quad\text{if and only if }
\left\{
\begin{array}{lllll}
  j=k=\frac{2}{\zeta}, & \text{when }n\geq 3,\\
  j=k=\frac{2}{\zeta}\text{ or }j=1,~k=0, & \text{when }n=2,\\
\end{array}
\right.
\]
and 
\[
2-\zeta k-n(j-k)\leq 2-\zeta,
\]
for all $j=1,2,\cdots,$ $k=0,\cdots,j$ with the equality holds if and only if $j=k=1$.
Consequently when $\zeta<2$, we have from asymptotic expansion \eqref{equ:radial} that there exist $C,R>0$ such that
\[
\left|
u(x)-\left(\frac{r^2}{2}+\frac{r^{2-\zeta}}{(2-\zeta)(n-\zeta)}\right)
\right|\leq |C_1|\ln r+|C_2|+Cr^{2-2\zeta}+Cr^{2-n},
\]
for all $r>R$. The desired estimates in \eqref{equ:example-result-1} and 
\eqref{equ:example-result-2} with $\zeta<2$ follow immediately and they are optimal in the sense that the order $r^{2-\zeta}$ cannot be smaller. 

When $\zeta=2<n$,  we have  from asymptotic expansion \eqref{equ:radial} that  $C_1=\frac{1}{n-\zeta}>0$ and there exist $C,R>0$ such that
\[
\left|
u(x)-\left(\frac{r^2}{2}+\frac{\ln r}{n-\zeta}-\frac{C_0r^{2-n}}{n-2}\right)
\right|\leq |C_2|+Cr^{-2},
\]
for all $r>R$. The desired estimate in \eqref{equ:example-result-1} with $\zeta=2$ follows immediately and it is optimal in the sense that the order $\ln r$ cannot be smaller. 

When $\zeta=n$, we have from asymptotic expansion \eqref{equ:radial-22} that there exist $C,R>0$ such that
\[
\left|
u(x)-\left(\frac{r^2}{2}+\frac{(\ln r)^2}{2}\right)
\right|\leq |C_3|\ln r+|C_4|+Cr^{-2}(\ln r)^2,
\]
for all $r>R$. The desired estimate in \eqref{equ:example-result-2} with $\zeta=2$ follows immediately and it is optimal in the sense that the order $(\ln r)^2$ cannot be smaller.
\end{proof}

\section{Quadratic term at infinity}\label{seclabel-QuadraticTerm}

In this section, we capture the quadratic term at infinity. Hereinafter, we let $B_r(x)$ denote the ball centered at $x$ with radius $r$ and $B_r:=B_r(0)$. Furthermore, by the interior regularity by Caffarelli \cite{Caffarelli-InteriorEstimates-MA} and Jian--Wang \cite{Jian-Wang-ContinuityEstimate-MA} and extension theorem of convex functions by Min \cite{Min-Extension-ConvexFunc}, we may assume without loss of generality that $f$ is strictly positive and $u$ is a classical solution.
\begin{theorem}\label{thm:quadratic-ori}
  Let $u \in C^{0}\left(\mathbb{R}^{n}\right)$ be a convex viscosity solution of \eqref{equ:MA} with $n\geq 2$ and $u(0)=\min_{\mathbb R^n}u=0$.
  Let $0<f\in C^0(\mathbb R^n)$ satisfy
  \begin{equation}\label{equ:cond-f-integ}
\left(
\int_{B_{R}}\left|f(z)^{\frac{1}{n}}-1\right|^{n}  \mathrm{d} z
\right)^{\frac{1}{n}}\leq CR^{1-\zeta}
  \end{equation}
  for some $C>0$ and $\zeta>0$.
  Then there exists a linear transform $T$ satisfying $\det T=1$ such that $v(x):=u(Tx)$ satisfies
  \begin{equation}\label{equ-roughConverge-ori}
\left|v(x)-\frac{1}{2}|x|^2\right|\leq C|x|^{2-\epsilon}\quad\forall~|x|\geq 1
\end{equation}
for some $C>0$ and $0<\epsilon<\min\{\frac{1}{10},\frac{\zeta}{3}\}$.
\end{theorem}

\begin{remark}
  If $f\in C^0(\mathbb R^n)$ satisfies
  \begin{equation}\label{equ:cond-f-0order}
  |f(x)-1|\leq C|x|^{-\zeta'}\quad\forall~x\in\mathbb R^n
  \end{equation}
  for some $C>0$ and $\zeta'>0$. Then \eqref{equ:cond-f-integ} holds with
  $\zeta=\zeta'$ or any $0<\zeta<1$ when $0<\zeta'<1$ or $\zeta'\geq 1$ respectively. In fact, by a direct computation, for all $R>2$ there exists $C>0$ (which may vary from term to term) such that
$$
\int_{B_{R}}\left|f(x)^{\frac{1}{n}}-1\right|^{n} \mathrm{d} x
\leq C\int_{B_R}|f(x)-1|^n  \mathrm{d} x\leq \left\{
\begin{array}{llll}
  CR^{n(1-\zeta')}, & \text{if }0<\zeta'<1,\\
  C\ln R, & \text{if }\zeta'=1,\\
  C, & \text{if }\zeta'>1.\\
\end{array}
\right.
$$
\end{remark}

Theorem \ref{thm:quadratic-ori} has been proved in Theorem 1.2 by Bao--Li--Zhang \cite{Bao-Li-Zhang-ExteriorBerns-MA} when $\zeta=1$ and it follows similarly from the proof therein (see also Proposition 3.3 of Caffarelli--Li \cite{Caffarelli-Li-ExtensionJCP}) by changing the $\epsilon$ from $\frac{1}{10}$ into $\min\{\frac{\epsilon}{3},\frac{1}{10}\}$.
Theorem \ref{thm:quadratic-ori}  proves  estimate \eqref{equ-roughConverge} by a change of variable.
\begin{proof}[Proof of \eqref{equ-roughConverge}]
  Let $u$ be as in Theorems \ref{thm:main-n>3} and \ref{thm:main-n=2}. Change of variable by setting
  $$
  \bar u(x):=\dfrac{1}{(f(\infty))^{\frac{1}{n}}}\left(u(x)-Du(0)\cdot x-u(0)\right)\quad\text{in }\mathbb R^n.
  $$
  By a direct computation, $\bar u$ satisfies Eq.~\eqref{equ:MA} with $f$ replaced by $\bar f(x):=\frac{f(x)}{f(\infty)}$. By taking $k=0$ in \eqref{equ:cond-f}, $\bar f$ verifies condition \eqref{equ:cond-f-0order} with $\zeta>0$ given in Theorems \ref{thm:main-n>3} and \ref{thm:main-n=2}. By Theorem \ref{thm:quadratic-ori}, there exists a linear transform $T$ with $\det T=1$ such that $\widetilde u(x):=\bar u(Tx)$ satisfies \eqref{equ-roughConverge-ori}. Since $T$ is invertible, we have

  $$
  \begin{array}{lllll}
    \displaystyle\left|\bar u(x)-\frac{1}{2}x^T(T^TT)x\right| & = &
    \displaystyle \left|\bar u(Tx)-\frac{1}{2}(Tx)^T(Tx)\right|\\
    &=&
    \displaystyle \left|\bar u(y)-\frac{1}{2}|y|^2\right|\\
    &\leq & C|y|^{2-\epsilon}\leq C|x|^{2-\epsilon}\\
  \end{array}
  $$
  for some $C>0$, where  $y:=Tx$. Then \eqref{equ-roughConverge} follows immediately by the definition of $\bar u$ and taking $A:=(f(\infty))^{\frac{1}{n}}T^TT>0$.
\end{proof}

\section{Preliminary on Poisson equations}\label{seclabel-Poissonequ}

In this section, we introduce the existence results for Poisson equation on exterior domain i.e.,
\begin{equation}\label{equ:Laplace}
  \Delta v=g\quad\text{in }\mathbb R^n\setminus\overline{B_1}.
\end{equation}

\begin{lemma}\label{Lem-existence-fastdecay}
  Let $g\in C^{\infty}(\mathbb R^n)$ with $n\geq 2$ satisfy
  \begin{equation}\label{Equ-cond-g}
  ||g(r\cdot)||_{L^p(\mathbb S^{n-1})}\leq c_0 r^{-k_1}(\ln r)^{k_2}\quad\forall~r>1
  \end{equation}
  for some $c_0>0,~k_1>0,~k_2\ge 0$ and $p>\frac{n}{2}, p\geq 2$. Then there exists a smooth solution
  $v$ of \eqref{equ:Laplace}
  such that
  \begin{equation}\label{Equ-exist-1}
  |v(x)|\leq
Cc_0|x|^{2-k_1}(\ln|x|)^{k},
  \end{equation}
  for some constant $C$ relying only on $n,k_1,k_2,p$ and
  \begin{equation}\label{equ:def-expone-k}
 k=\left\{
  \begin{array}{lllll}
  {k_2}, & \text{if }k_1\not\in\mathbb N, & \text{and }n=2,\\
 {k_2+1}, & \text{if }k_1\in\mathbb N\setminus\{2\},& \text{and }n=2,\\
 {k_2+2}, & \text{if }k_1=2,& \text{and }n=2,\\
  {k_2}, & \text{if }k_1-n\not\in\mathbb N, k_1\not\in\{1,2\},& \text{and }n\geq 3,\\
  {k_2+1}, & \text{if }k_1-n\in\mathbb N\text{ or }k_1\in\{1,2\},& \text{and }n\geq 3.\\
  \end{array}
  \right.
  \end{equation}
\end{lemma}

\begin{proof} The result on $n=2$  can be found in Lemma 2.1 in \cite{Liu-Bao-2021-Dim2-MeanCur} and the result on $n\geq 3$ with $k_1>2$ can be found in Lemma 3.1 in \cite{Liu-Bao-2020-ExpansionSPL}. Hence we only need to prove for $n\geq 3$ and $0<k_1\leq 2$ case.

Let $\Delta_{\mathbb{S}^{n-1}}$ be the Laplace--Beltrami operator on unit sphere $\mathbb{S}^{n-1}\subset\mathbb{R}^n$ and
$$
\Lambda_0=0,~\Lambda_1=n-1,~\Lambda_2=2n,~\cdots,~\Lambda_k=k(k+n-2),~\cdots,
$$
be the sequence of eigenvalues of $-\Delta_{\mathbb S^{n-1}}$ with eigenfunctions
\begin{equation*}Y_1^{(0)}=1,~Y_{1}^{(1)}(\theta),~Y_{2}^{(1)}(\theta),~\cdots,~ Y_{n}^{(1)}(\theta),~\cdots,~Y_{1}^{(k)}(\theta),~\cdots,~Y_{m_k}^{(k)}(\theta),~\cdots,
\end{equation*}
i.e.,
$$
-\Delta_{\mathbb{S}^{n-1}}Y_m^{(k)}(\theta)=\Lambda_kY_m^{(k)}(\theta)\quad\forall~
m=1,2,\cdots,m_k.
$$
The family of eigenfunctions forms a complete standard orthogonal basis of $L^2(\mathbb{S}^{n-1})$.

Expand $g$ and the wanted solution $v$ into
\begin{equation}\label{equ-star}
v(x)=\sum_{k=0}^{+\infty}\sum_{m=1}^{m_{k}} a_{k, m}(r) Y_{m}^{(k)}(\theta)\quad\text{and}\quad
g(x)=\sum_{k=0}^{+\infty}\sum_{m=1}^{m_{k}} b_{k, m}(r) Y_{m}^{(k)}(\theta),
\end{equation}
where
 $r=|x|, \theta=\frac{x}{|x|}$ and

 $$a_{k,m}(r):=\int_{\mathbb{S}^{n-1}} v(r \theta) \cdot Y_{m}^{(k)}(\theta) \mathrm d  \theta,\quad b_{k,m}(r):=\int_{\mathbb{S}^{n-1}} g(r\theta) \cdot Y_{m}^{(k)}(\theta)  \mathrm d  \theta.$$
In spherical coordinates,

$$
\Delta v=\partial_{rr}v+\dfrac{n-1}{r}\partial_rv+\dfrac{1}{r^2}\Delta_{\mathbb{S}^{n-1}}v
$$
and \eqref{equ:Laplace} becomes

$$
\sum_{k=0}^{\infty} \sum_{m=1}^{m_{k}}\left(a_{ k,m}^{\prime \prime}(r)+\frac{n-1}{r} a_{ k,m}^{\prime}(r)-\frac{\Lambda_{k}}{r^{2}} a_{k,m}(r)\right) Y_{m}^{(k)}(\theta)=
\sum_{k=0}^{+\infty} \sum_{m=1}^{m_{k}} b_{k, m}(r) Y_{m}^{(k)}(\theta).
$$
By the linearly independence of eigenfunctions, for all $k\in\mathbb{N}$ and $m=1,2,\cdots,m_k$,
\begin{equation}\label{Equ-equ-equ}
a_{k,m}^{\prime \prime}(r)+\frac{n-1}{r} a_{ k,m}^{\prime}(r)-\frac{\Lambda_{k}}{r^{2}} a_{k,m}(r) =b_{k,m}(r)\quad\text{in }r>1.
\end{equation}

By solving the ODE, there exist constants $C_{k,m}^{(1)},C_{k,m}^{(2)}$ such that for all $r>1$,
\begin{equation}\label{Equ-def-Wronski}
\begin{array}{lll}
  a_{k,m}(r)&=&C_{k,m}^{(1)}r^k+
C_{k,m}^{(2)}r^{2-n-k}\\
&&\displaystyle
-\dfrac{1}{2-n}r^k\int_{2}^r\tau^{1-k}b_{k,m}(\tau) \mathrm d  \tau
+\dfrac{1}{2-n}r^{2-k-n}\int_{2}^r\tau^{k+n-1}b_{k,m}(\tau) \mathrm d  \tau.
\end{array}
\end{equation}
 By \eqref{Equ-cond-g},
\begin{equation}\label{Equ-converge}
\sum_{k=0}^{+\infty}\sum_{m=1}^{m_k}|b_{k,m}(r)|^2=||g(r\cdot)||^2_{L^2(\mathbb{S}^{n-1})}
\leq c_0^2\omega_n^{\frac{p-2}{p}}r^{-2k_1}(\ln r)^{2k_2}
\end{equation}
for all $r>1$. Then

\begin{equation}\label{equ:integrable-cond}
r^{1-k}b_{k,m}(r)\in L^1(2,+\infty)\quad\text{for }\left\{
\begin{array}{llll}
  k\geq 1, & \text{if }1<k_1\leq 2,\\
  k\geq 2, & \text{if }0<k_1\leq 1.\\
\end{array}
\right.
\end{equation}
We choose $C_{k,m}^{(1)}$ and $C_{k,m}^{(2)}$ in \eqref{Equ-def-Wronski} such that
\begin{equation}\label{Equ-def-v-2}
a_{k,m} (r):=
- \dfrac{1}{2-n}r^k\int_{+\infty}^r\tau^{1-k}b_{k,m}(\tau) \mathrm d  \tau
+ \dfrac{1}{2-n}r^{2-k-n}\int_{2}^r\tau^{k+n-1}b_{k,m}(\tau) \mathrm d  \tau
\end{equation}
for all $k$ verifying \eqref{equ:integrable-cond} and
\begin{equation}\label{Equ-def-v-3}
a_{k,m} (r):=
- \dfrac{1}{2-n}r^k\int_{2}^r\tau^{1-k}b_{k,m}(\tau) \mathrm d  \tau
+ \dfrac{1}{2-n}r^{2-k-n}\int_{2}^r\tau^{k+n-1}b_{k,m}(\tau) \mathrm d  \tau
\end{equation}
for all rest $k$.

For $1<k_1<2$ case,
we may pick $0<\epsilon:=\frac{1}{2}\min\{1,\mathtt{dist}(k_1,\mathbb N)\}$ such that

$$
\left\{
\begin{array}{lllll}
  3-2k_1-\epsilon >-1,\\
  3-2k-2k_1+\epsilon<-1, & \text{for }k\geq 1,\\
  2k+2n-2k_1-1-\epsilon>-1, & \text{for }k\geq 1.\\
\end{array}
\right.
$$
Then by \eqref{Equ-converge} and H\"older inequality, we have

\begin{align*}
&\displaystyle a_{0,1}^2(r)+\sum_{k=1}^{+\infty}\sum_{m=1}^{m_k}a_{k,m}^2(r)\\
  \leq & \displaystyle  2\left|\int_{2}^{r} \tau^{1} b_{0,1}(\tau) \mathrm d \tau\right|^2+2\left|r^{2-n}\int_{2}^r\tau^{n-1} b_{0,1}(\tau)\mathrm d\tau\right|^2
  +2\sum_{k=1}^{+\infty}\sum_{m=1}^{m_k}r^{2k}\left|
  \int_r^{+\infty}\tau^{1-k}b_{k,m}(\tau)\mathrm d\tau
  \right|^2\\
&\displaystyle +
 2\sum_{k=1}^{+\infty}\sum_{m=1}^{m_k}r^{2(2-k-n)}\left|
  \int_{2}^{r} \tau^{k+n-1} b_{k, m}(\tau) \mathrm  d  \tau
  \right|^2\\
\leq &\displaystyle 2\int_2^r\tau^{3-2k_1-\epsilon}(\ln\tau)^{2k_2}\mathrm d\tau\cdot \int_2^r\tau^{2k_1}(\ln\tau)^{-2k_2}b_{0,1}^2(\tau)\frac{\mathrm d\tau}{\tau^{1-\epsilon}}\\
&\displaystyle +2r^{4-2n}\int_2^r\tau^{2n-2k_1-1-\epsilon}(\ln\tau)^{2k_2}\mathrm d\tau\cdot \int_2^r\tau^{2k_1}(\ln\tau)^{-2k_2}b_{0,1}^2(\tau)\frac{\mathrm d\tau}{\tau^{1-\epsilon}}\\
&\displaystyle +2\sum_{k=1}^{+\infty}\sum_{m=1}^{m_k}r^{2k}
\int_r^{+\infty}\tau^{3-2k-2k_1+\epsilon}(\ln\tau)^{2k_2}\mathrm d\tau
\int_r^{+\infty}\tau^{2k_1}(\ln\tau)^{-2k_2}b_{k,m}^2(\tau)\frac{\mathrm d\tau}{\tau^{1+\epsilon}}\\
&\displaystyle +2\sum_{k=1}^{+\infty}\sum_{m=1}^{m_k} r^{4-2k-2n}
\int_2^r\tau^{2k+2n-2k_1-1-\epsilon}(\ln\tau)^{2k_2}\mathrm  d\tau
\int_2^r\tau^{2k_1}(\ln\tau)^{-2k_2}b_{k,m}^2(\tau)\frac{\mathrm d\tau}{\tau^{1-\epsilon}}\\
\leq &\displaystyle Cr^{4-2k_1-\epsilon}(\ln r)^{2k_2} \int_2^r\tau^{2k_1}(\ln\tau)^{-2k_2}b_{k,m}^2(\tau)\frac{\mathrm d\tau}{\tau^{1-\epsilon}}\\
&+\displaystyle C\sum_{k=1}^{+\infty}\sum_{m=1}^{m_k} r^{4-2k_1+\epsilon}(\ln\tau)^{2k_2}
\int_r^{+\infty}\tau^{2k_1}(\ln\tau)^{-2k_2}b_{k,m}^2(\tau)\frac{\mathrm d\tau}{\tau^{1+\epsilon}}\\
&+\displaystyle C\sum_{k=1}^{+\infty}\sum_{m=1}^{m_k}
r^{4-2k_1-\epsilon}(\ln\tau)^{2k_2}
\int_2^{r}\tau^{2k_1}(\ln\tau)^{-2k_2}b_{k,m}^2(\tau)\frac{\mathrm d\tau}{\tau^{1-\epsilon}}\\
\leq & \displaystyle
Cr^{4-2k_1-\epsilon}(\ln r)^{2k_2}\int_2^r\tau^{2k_1}(\ln\tau)^{-2k_2}
\sum_{k=0}^{+\infty}\sum_{m=1}^{m_k}b_{k,m}^2(\tau)\frac{\mathrm d\tau}{\tau^{1-\epsilon}}\\
&+\displaystyle
Cr^{4-2k_1+\epsilon}(\ln r)^{2k_2} \int_r^{+\infty}
\tau^{2k_1}(\ln\tau)^{-2k_2}\sum_{k=0}^{+\infty}\sum_{m=1}^{m_k} b_{k,m}^2(\tau)\frac{\mathrm d\tau}{\tau^{1+\epsilon}}\\
\leq & Cc_0^2\cdot r^{4-2k_1}(\ln r)^{2k_2}.\\
\end{align*}

For $0<k_1<1$ case, we may pick $0<\epsilon:=\frac{1}{2}\min\{1,\mathtt{dist}(k_1,\mathbb N)\}$ such that

$$
\left\{
\begin{array}{lllll}
  1-2k_1-\epsilon>-1,\\
  3-2k-2k_1+\epsilon<-1, & \text{for }k\geq 2,\\
  2k+2n-2k_1-1-\epsilon>-1, & \text{for }k\geq 1,\\
\end{array}
\right.
$$
and change $a_{1,m}$ with $m=1,2,\cdots,n$ into
\eqref{Equ-def-v-3}.
The estimates of $\displaystyle a_{0,1}^2(r)+\sum_{k=1}^{+\infty}\sum_{m=1}^{m_k}a_{k,m}^2(r)$ follow similarly.

For $k_1=2$ case, we may pick $0<\epsilon:=\frac{1}{2} $ such that

$$
\left\{
\begin{array}{lllll}
  2n-2k_1-1-\epsilon>-1,\\
  3-2k-2k_1+\epsilon<-1, & \text{for }k\geq 1,\\
  2k+2n-2k_1-1-\epsilon>-1, & \text{for }k\geq 1,\\
\end{array}
\right.
$$
and use the following estimates of $a_{0,1}^2$.

$$
\begin{array}{llll}
  a_{0,1}^2(r) & \leq & \displaystyle    2\left|\int_{2}^{r} \tau^{1} b_{0,1}(\tau) \mathrm d \tau\right|^2+2\left|r^{2-n}\int_{2}^r\tau^{n-1} b_{0,1}(\tau)\mathrm  d\tau\right|^2\\
&\leq &\displaystyle 2\int_2^r\tau^{-1}(\ln\tau)^{2k_2} \mathrm d\tau\cdot \int_2^r\tau^{2k_1}(\ln\tau)^{-2k_2}b_{0,1}^2(\tau)\frac{\mathrm  d\tau}{\tau}\\
&&\displaystyle +2r^{4-2n}\int_2^r\tau^{2n-2k_1-1-\epsilon}(\ln\tau)^{2k_2}\mathrm  d\tau\cdot \int_2^r\tau^{2k_1}(\ln\tau)^{-2k_2}b_{0,1}^2(\tau)\frac{ \mathrm d\tau}{\tau^{1-\epsilon}}\\
&\leq& \displaystyle  Cc_0^2\cdot r^{4-2k_1}(\ln r)^{2k_2+2}.\\
\end{array}
$$
The rest parts of estimate follow similarly.

For $k_1=1$ case, we may pick $0<\epsilon:=\frac{1}{2} $ such that

$$
\left\{
\begin{array}{lllll}
  3-2k_1-\epsilon >-1,\\
  3-2k-2k_1+\epsilon<-1, & \text{for }k\geq 2,\\
  2k+2n-2k_1-1-\epsilon>-1, & \text{for }k\geq 1,\\
\end{array}
\right.
$$
and we use the following estimates of $\displaystyle \sum_{m=1}^{m_1}a^2_{1,m}$.
\begin{align*}
 \displaystyle \sum_{m=1}^{m_1}a_{1,m}^2(r) &\leq  \displaystyle
  2\sum_{m=1}^{m_1} r^{2}\left|
\int_{2}^rb_{1,m}(\tau) \mathrm  d\tau
  \right|^2+ 2\sum_{m=1}^{m_1}r^{2-2n}\left|
  \int_{2}^{r} \tau^{n} b_{1, m}(\tau) \mathrm  d  \tau
  \right|^2\\
&\leq  \displaystyle 2\sum_{m=1}^{m_1} r^{2}\int_2^r\tau^{-1}(\ln\tau)^{2k_2} \mathrm d\tau
\int_2^r\tau^{2k_1}(\ln\tau)^{-2k_2}b_{1,m}^2(\tau)
\frac{\mathrm d\tau}{\tau}\\
& \displaystyle +2\sum_{m=1}^{m_1} r^{2-2n}
\int_2^r\tau^{2n-2k_1+1-\epsilon}(\ln\tau)^{2k_2} \mathrm  d\tau
\int_2^r\tau^{2k_1}(\ln\tau)^{-2k_2}b_{1,m}^2(\tau)\frac{\mathrm d\tau}{\tau^{1-\epsilon}}\\
&\leq   Cc_0^2\cdot r^{4-2k_1}(\ln r)^{2k_2+2}.\\
\end{align*}
The rest parts of estimate follow similarly.

Consequently, $v(r)$ is well-defined, is a solution of \eqref{equ:Laplace} in distribution sense \cite{Gunther-ConformalNormalCoord} and satisfies
\begin{equation}\label{equ-estimate-spherical}
\begin{array}{llllll}
||v(r\cdot)||^2_{L^2(\mathbb S^{n-1})}& \leq& \left\{
\begin{array}{lll}
  C c_0^2 \cdot r^{4-2k_1}(\ln r)^{2k_2}, & k_1\not\in\{1,2\},\\
  C c_0^2 \cdot r^{4-2k_1}(\ln r)^{2k_2+2}, & k_1\in\{1,2\}.
\end{array}
\right.\\
\end{array}
\end{equation}
By interior regularity theory of elliptic differential equations,  $v$ is smooth \cite{Book-Gilbarg-Trudinger}. It remains to prove the pointwise decay rate at infinity.

For any $r\gg 1$, we set

$$
v_r(x):=v(rx)\quad\forall~x\in B_4\setminus B_{1}=:D.
$$
Then $v_r$ satisfies
\[
\Delta v_r=r^2g(rx)=:g_r(x)\quad\text{in}~D.
\]
By weak Harnack inequality (see for instance Theorem 8.17 of \cite{Book-Gilbarg-Trudinger}, see also (2.11) of \cite{Gunther-ConformalNormalCoord}),

$$
\sup_{2<|x|<3}|v_r(x)|\leq C(n,p)\cdot \left(
||v_r||_{L^2(D)}+||g_r||_{L^p(D)}
\right).
$$
By \eqref{equ-estimate-spherical},
\begin{align*}
  ||v_r||_{L^2(D)}^2 & =  \displaystyle  \dfrac{1}{r^n}\int_{B_{4r}\setminus B_{r}}|v(x)|^2\mathrm  d  x\\
  &= \displaystyle r^{-n}\int_{ r}^{4r}||v(\tau\theta)||_{L^2(\mathbb S^{n-1})}^2\cdot \tau^{n-1}\mathrm  d  \tau\\
  &\leq
  \left\{
  \begin{array}{lll}
  \displaystyle Cc_0^2\cdot r^{-n}\int_{ r}^{4r}
  \tau^{4-2k_1}(\ln\tau)^{2k_2}\cdot \tau^{n-1}\mathrm  d  \tau,
  & k_1\not\in\{1,2\},\\
  \displaystyle Cc_0^2\cdot r^{-n}\int_{ r}^{4r}
  \tau^{4-2k_1}(\ln\tau)^{2k_2+2}\cdot \tau^{n-1} \mathrm d  \tau,
  & k_1\in\{1,2\},\\
  \end{array}
  \right.\\
  &\leq \left\{
  \begin{array}{llll}
    Cc_0^2\cdot  r^{4-2k_1}(\ln r)^{2k_2},  & k_1\not\in\{1,2\}\\
    Cc_0^2\cdot  r^{4-2k_1}(\ln r)^{2k_2+2}, & k_1\in\{1,2\}.\\
  \end{array}
  \right.
\end{align*}
By \eqref{Equ-cond-g},
\[
\begin{array}{llll}
  ||g_r||_{L^p(D)}^p & =\displaystyle
  \dfrac{r^{2p}}{r^n}\int_{B_{4r}\setminus B_{ r}}|g(x)|^p \mathrm d  x\\
  &\leq  \displaystyle Cc_0^p \cdot r^{2p-n}\int_{ r}^{4r}\tau^{-pk_1}(\ln\tau)^{pk_2}\cdot \tau^{n-1} \mathrm  d  \tau\\
  &\leq Cc_0^p\cdot r^{2p-pk_1}(\ln r)^{pk_2}.\\
\end{array}
\]
Combining the estimates above, we have
\[
\begin{array}{llll}
&\displaystyle \sup_{2r<|x|<3r}|v(x)| =
\sup_{2<|x|<3}|v_r(x)|\\
\leq & \displaystyle
\left\{
\begin{array}{llll}
  Cc_0r^{2-k_1}(\ln r)^{k_2}+Cc_0r^{2-k_1}(\ln r)^{k_2},
  & k_1\not\in\{1,2\},\\
  Cc_0r^{2-k_1}(\ln r)^{k_2+1}+Cc_0r^{2-k_1}(\ln r)^{k_2},
  & k_1\in\{1,2\},\\
\end{array}
\right.\\
\end{array}
\]
where $C$ relies only on $n,k_1,k_2$ and $p$.
This finishes the proof of Lemma \ref{Lem-existence-fastdecay}.
\end{proof}

Similar to Lemma 3.2 in \cite{Liu-Bao-2020-ExpansionSPL}, by interior estimate we have the vanishing speed for higher order derivatives as below.

\begin{lemma}\label{lem-existence-fastd-deri}
  Let $g\in C^{\infty}(\mathbb R^n)$ satisfy
\[
  g(x)=O_{l}(|x|^{-k_1}(\ln |x|)^{k_2})\quad\text{as}~|x|\rightarrow+\infty
\]
  for some $k_1>0, k_2\geq 0$, $l-1\in\mathbb N$.   Then
\[
v_g(x)=O_{l+1}(|x|^{2-k_1}(\ln|x|)^k),
\]
  where $v_g$ denotes the solution found in Lemma \ref{Lem-existence-fastdecay} and $k$ is as in \eqref{equ:def-expone-k}.
\end{lemma}

\section{Proof for $n\geq 3$ case}\label{seclabel-n>3}

In this section, we prove Theorem \ref{thm:main-n>3}. By Theorem 3.1 and Remark 3.3 in \cite{Liu-Bao-2021-Expansion-LagMeanC} (see also Corollary 2.1 in \cite{Li-Li-Yuan-BernsteinThm} or Theorem 2.2 in \cite{Jia-Xiaobiao-AsymGeneralFully}), we have the following result on linear elliptic equations.

\begin{theorem}\label{thm:linearLimit}
  Let $v$ be a classical solution of
  \begin{equation}\label{Dirichlet}
   a_{i j}(x) D_{i j} v=f(x)\quad\text{in }\mathbb R^n,
  \end{equation}
  that is bounded from at least one side or $|Dv(x)|=O(|x|^{-1})$ as $|x|\rightarrow\infty$,
 where $n\geq 3$,  the coefficients are  uniformly elliptic,
  satisfying $
  ||a_{ij}||_{C^{\alpha}(\mathbb{R}^n)}<\infty$
  for some $0<\alpha <1$
  and
  \begin{equation}\label{short-RangeCoefficient}
  a_{ij}(x)=a_{ij}(\infty)+O(|x|^{-\varepsilon})\quad\text{as }|x|\rightarrow\infty,
  \end{equation}
  for some $\varepsilon>0$ and  $0<[a_{ij}(\infty)]\in\mathtt{Sym}(n)$.
  Hereinafter, $[a_{ij}]$ denote the $n$ by  $n$ matrix with the $i,j$-position being $a_{ij}$.
  Assume that $f\in C^0(\mathbb R^n)$ satisfies
  \begin{equation}
  \label{decayoff_1}
 f(x)=O(|x|^{-\zeta})\quad\text{as }|x|\rightarrow\infty,
  \end{equation}
  for some $\zeta>2$.
  Then there exists a constant $v_{\infty}$ such that \begin{equation}\label{Result_ExteriorLiouville-2}
  v ( x ) = v _ { \infty } +
  \left\{
  \begin{array}{llll}
  O \left( |x|^{2-\min\{n,\zeta\}} \right), & \zeta\not=n,\\
  O \left( |x|^{2-n}(\ln|x|) \right), & \zeta=n,\\
  \end{array}
  \right.
  \end{equation}
  as $|x|\rightarrow\infty$.
\end{theorem}
\begin{lemma}\label{lem:initial-n>3}
  Let $u,f$ be as in Theorem \ref{thm:main-n>3} and $A,\epsilon$ be as in \eqref{equ-roughConverge}. Then there exists $\alpha>0$ such that for some $C>0$,
  \begin{equation}\label{equ:boundedC2Alpha}
  ||D^2u||_{C^{\alpha}(\mathbb R^n)}\leq C
  \end{equation}
  and
  \begin{equation}\label{equ:rough-derivat}
  u(x)-\frac{1}{2}x^TAx=O_{m+1}(|x|^{2-\epsilon})\quad\text{as }|x|\rightarrow\infty.
  \end{equation}
\end{lemma}
\begin{proof}
  As proved in section \ref{seclabel-QuadraticTerm}, there exist $A,\epsilon$ such that \eqref{equ-roughConverge} holds.
  For sufficiently large $|x|>2$, set $R:=|x|$ and
  \begin{equation*}
u_{R}(y)=\left(\frac{4}{R}\right)^{2} u\left(x+\frac{R}{4} y\right), \quad|y| \leq 2.
\end{equation*}
Then by \eqref{equ-roughConverge}, there exists $C>0$ uniform to $R>2$ such that $$
||u_R||_{C^0(\overline{B_2})}\leq C.
$$
By a direct computation, $u_R$ satisfies
\begin{equation}\label{equ:scaled}
 \det D^2u_R(y)=f\left(x+\frac{R}{4}y\right)=:f_R(y)\quad\text{in }B_2.
\end{equation}
By taking $k=0,1$ in condition \eqref{equ:cond-f}, there exists $C>0$ uniform to $R>2$ such that
$$
\left\Vert f_R
-f(\infty)
\right\Vert_{C^0(\overline{B_2})}\leq CR^{-\zeta},
$$
and for any $0<\alpha<1$ and $y_1,y_2\in B_2$, $y_1\not=y_2$,
\begin{equation*}
\frac{\left|f_{R}\left(y_{1}\right)-f_{R}\left(y_{2}\right)\right|}{\left|y_{1}-y_{2}\right|^{\alpha}}=\frac{\left|f\left(z_{1}\right)-f\left(z_{2}\right)\right|}{\left|z_{1}-z_{2}\right|^{\alpha}} \cdot\left(\frac{R}{4}\right)^{\alpha} \leq C R^{-\zeta},
\end{equation*}
where $z_i:=x+\frac{R}{4}y_i\in B_{\frac{|x|}{2}}(x)$ for $i=1,2$. By the interior estimates by Caffarelli \cite{Caffarelli-InteriorEstimates-MA} and Jian--Wang \cite{Jian-Wang-ContinuityEstimate-MA} (see also the appendix in \cite{Bao-Li-Zhang-ExteriorBerns-MA}), we have
\begin{equation}\label{equ:boundScleC2Alpha}
\left\Vert
D^2u_R
\right\Vert_{C^{\alpha}(\overline{B_1})}
\leq C\left(1+\dfrac{||f_R||_{C^{\alpha}(B_2)}}{\alpha(1-\alpha)}\right)\leq C
\end{equation}
for some $C>0$ uniform to $R>2$. This yields \eqref{equ:boundedC2Alpha} by a direct computation.

Let
$$
v(x):=u(x)-\frac{1}{2}x^TAx\quad\text{and}\quad v_{R}(y):=\left(\frac{4}{R}\right)^{2} v\left(x+\frac{R}{4} y\right), \quad|y| \leq 2,
$$
where $R=|x|>2$ as above. Then by \eqref{equ-roughConverge} there exists $C>0$ uniform to $|x|>2$ such that
\begin{equation*}
\left\|v_{R}\right\|_{C^{0}\left(\overline{B_{2}}\right)} \leq C R^{-\varepsilon}.
\end{equation*}
Hereinafter, we set $F(M):=\det M$, $D_{M_{ij}}F(M)$ denote the partial derivative of $F$ with respect to $M_{ij}$ position and  $D_{M_{ij},M_{kl}}F(M)$ denote the partial derivative of $F$ with respect to $M_{ij}, M_{kl}$ positions.
Applying Newton--Leibnitz formula between \eqref{equ:scaled} and $F(A)=f(\infty)$, we have
\begin{equation}\label{equ:temp}
\widetilde{a_{ij}}(y)D_{ij}v_R=f_R(y)-f(\infty)\quad\text{in }B_2,
\end{equation}
where $\widetilde{a_{ij}}(y)=\int_0^1D_{M_{ij}}F(A+tD^2v_R(y)) \mathrm{d} t$. Since $F$ is smooth, by \eqref{equ:boundScleC2Alpha} we have $C>1$ uniform to $R>2$ such that
\begin{equation*}
\frac{I}{C} \leq \widetilde{a_{i j}} \leq C I \quad \text {in } B_{1} \quad\text{and}\quad\left\|\widetilde{a_{i j}}\right\|_{C^{\alpha}\left(\overline{B_{1}}\right)} \leq C.
\end{equation*}

By interior Schauder estimates as Theorem 6.2 of \cite{Book-Gilbarg-Trudinger}, we have
\begin{equation}\label{equ-interiorSchauder}
\begin{array}{llll}
\left\|v_{R}\right\|_{C^{2, \alpha}\left(\overline{B_{\frac{1}{2}}}\right)} &\leq & C\left(\left\|v_{R}\right\|_{C^0\left(\overline{B_{1}}\right)}+\left\|f_{ R}-f(\infty)\right\|_{C^{\alpha}\left(\overline{B_{1}}\right)}\right)\\
& \leq & C R^{-\min\{\varepsilon,\zeta\}}.\\
\end{array}
\end{equation}
Higher order derivative estimates follow by further differentiating the equation and interior Schauder estimates.  More rigorously, for any $e\in\partial B_1$, by taking partial derivative to $F(A+D^2v_R(y))=f_R(y)$, we have
\begin{equation}\label{equ:temp-2}
\widehat{a_{ij}}(y)D_{ije}v_R=D_ef_R(y)\quad\text{in }B_2,
\end{equation}
where $\widehat{a_{ij}}(y)=D_{M_{ij}}F(A+D^2v_R(y))$.
By condition \eqref{equ:cond-f},
\[
||f_R-f(\infty))||_{C^{k,\alpha}(\overline{B_1})}\leq CR^{-\zeta},\quad\forall~k=0,1,\cdots,m-1.
\]
By \eqref{equ:boundScleC2Alpha}, since  $F$ is smooth and uniformly elliptic, we may apply interior Schauder estimate to \eqref{equ:temp-2} and obtain
\[
||D_ev_R||_{C^{2,\alpha}(\overline{B_{\frac{1}{4}}})}\leq CR^{-\min\{\epsilon,\zeta\}}.
\]
By taking partial derivative once again,
\[
\hat{a_{ij}}(y)D_{ijee}v_R+D_{M_{ij},M_{kl}}F(A+D^2v_R(y))D_{ije}v_RD_{kle}v_R=D_{ee}f_R(y)\quad\text{in }B_2.
\]
Since $F$ is smooth, condition \eqref{equ:cond-f} and the estimate above provides
$$
\left\Vert
D_{ee}f_R-D_{M_{ij},M_{kl}}F(A+D^2v_R(y))D_{ije}v_RD_{kle}v_R
\right\Vert_{C^{\alpha}(\overline{B_{\frac{1}{4}}})}\leq CR^{-\min\{\epsilon,\zeta\}}
$$
for some $C>0$ for all $R>2$.
By taking further derivatives and iterate, for all $k=0,1,\cdots,m+1$, there exists $C>0$ such that
\[
|D^{k}v_R(0)|\leq CR^{-\min\{\epsilon,\zeta\}}
\]
for all $R>2$.
From the proof in Theorem \ref{thm:quadratic-ori}, we have $\epsilon<\frac{\zeta}{3}<\zeta$ and then \eqref{equ:rough-derivat} follows by scaling back.
\end{proof}

Now we are ready to prove Theorem \ref{thm:main-n>3}.
\begin{proof}[Proof of Theorem \ref{thm:main-n>3}]
  Applying Newton--Leibnitz formula between Eq.~\eqref{equ:MA} and $\det A=f(\infty)$, we obtain  that $v:=u-\frac{1}{2}x^TAx$ satisfies
\begin{equation}\label{linearized-equation-3}
 \overline{a_{i j}}(x) D_{i j} v:=\int_{0}^{1} D_{M_{i j}} F\left(t D^{2} v+A\right) \mathrm{d} t \cdot D_{i j} v
 =f(x)-f(\infty)=:
\overline{f}(x).
\end{equation}
For any $e\in\partial B_1$, by the concavity of operator $F$, we act partial derivative $D_e$ and $D_{ee}$ to Eq.~\eqref{equ:MA} and obtain
\begin{equation}\label{linearized-equation-1}
  \widehat{a_{ij}}(x)D_{ije}v:=D_{M_{i j}} F\left(  D^{2} v+A\right) D_{i je} v=D_ef(x),
\end{equation}
and
\begin{equation}\label{linearized-equation-2}
\widehat{a_{ij}}(x)D_{i jee} v \geq D_{e e}f(x).
\end{equation}

By \eqref{equ:boundedC2Alpha} and \eqref{equ:rough-derivat} from Lemma \ref{lem:initial-n>3}, we have $C>0$ such that
\[
\left| \overline { a _ { i j } } ( x ) - D_{M_{ij}}F  (A ) \right| +\left| \widehat { a _ { i j } } ( x ) - D_{M_{ij}}F (A) \right| \leq \frac { C } { | x | ^ { \epsilon} }.
\]
By condition \eqref{equ:cond-f}, we have $D_{ee}f=O(|x|^{-2-\zeta})$ as $|x|\rightarrow+\infty$. By constructing barrier functions (see for instance \cite{Li-Li-Yuan-BernsteinThm,Jia-Xiaobiao-AsymGeneralFully}), there exists $C>0$ such that for all $x\in\mathbb R^n$,
\[
D_ { e e }v ( x ) \leq
\left\{
\begin{array}{llll}
C | x | ^ {2-\min\{n,\zeta+2\}}, & \zeta\not=n-2,\\
C|x|^{2-n}(\ln|x|), & \zeta=n-2.\\
\end{array}
\right.
\]
By the arbitrariness of $e$,
\[
\lambda_{\max }\left(D^{2} v\right)(x) \leq \left\{
\begin{array}{llll}
C | x | ^ {2-\min\{n,\zeta+2\}}, & \zeta\not=n-2,\\
C|x|^{2-n}(\ln|x|), & \zeta=n-2.\\
\end{array}
\right.
\]
By condition \eqref{equ:cond-f} and the ellipticity of Eq.~ (\ref{linearized-equation-3}),
\[
\lambda_{\min }\left(D^{2} v\right)(x) \geq-C \lambda_{\max }\left(D^{2} v\right)-C|\overline{f}(x)| \geq
\left\{
\begin{array}{llll}
-C | x | ^ {2-\min\{n,\zeta+2\}}, & \zeta\not=n-2,\\
-C|x|^{2-n}(\ln|x|), & \zeta=n-2.\\
\end{array}
\right.
\]
Hence  there exists $C>0$ such that

\begin{equation}\label{equ-VaniRate-Hessian}
\left| D ^ { 2 } v ( x ) \right| \leq
\left\{
\begin{array}{llll}
C | x | ^ {2-\min\{n,\zeta+2\}}, & \zeta\not=n-2,\\
C|x|^{2-n}(\ln|x|), & \zeta=n-2.\\
\end{array}
\right.\end{equation}
Rewrite \eqref{linearized-equation-3} into
$$
D_{M_{ij}}F(A)D_{ij}v=\overline f(x)+\left(
\overline{a_{ij}(x)}-D_{M_{ij}}F(A)
\right)D_{ij}v(x)=:g(x)
$$
in $\mathbb R^n$.
Let
\begin{equation}\label{equ:defni-Q}
Q:=
[D_{M_{i j}} F(A)]^{\frac{1}{2}}\quad\text{and}\quad \widetilde v(x):=v(Qx).
\end{equation}
Since trace is invariant under cyclic permutations, we have
\begin{equation}\label{temp-1}
\Delta \widetilde v(x)=g(Qx)=:\widetilde g(x)\quad\text{in } \mathbb R^n.
\end{equation}

If $0<\zeta\leq 1$, then \eqref{equ-VaniRate-Hessian} becomes
$$
\left|D^2v(x)\right|=
\left\{
\begin{array}{llll}
O( | x | ^ {-\zeta}), & \text{if }0<\zeta<1,\\
& \text{or }\zeta=1\text{ and }n\geq4,\\
O(|x|^{-1}(\ln|x|)), & \text{if }\zeta=1\text{ and }n=3.\\
\end{array}
\right.
$$
By a direct computation, it yields
$$
|v(x)|, |\widetilde v(x)|=
\left\{
\begin{array}{llll}
O_2(| x | ^ {2-\zeta}), & \text{if }0<\zeta<1,\\
O_2(|x|(\ln|x|)),& \text{if }\zeta=1\text{ and }n\geq4,\\
O_2(|x|(\ln|x|)^2), & \text{if }\zeta=1\text{ and }n=3.\\
\end{array}
\right.
$$
By letting $v_R$ as in the proof of Lemma \ref{lem:initial-n>3} and applying interior Schauder estimates, we have
$$
||D_ev_R||_{C^{2,\alpha}(\overline{B_{\frac{1}{4}}})}\leq
\left\{
\begin{array}{lllll}
CR^{-\zeta}, & \text{if }0<\zeta<1,\\
CR^{-1}(\ln R), & \text{if }\zeta=1\text{ and }n\geq 4,\\
CR^{-1}(\ln R)^2, & \text{if }\zeta=1\text{ and }n=3,\\
\end{array}
\right.
$$
for some $C>0$ for all $R>2$ and $e\in\partial B_1$. By applying higher order derivatives to \eqref{equ:temp-2} and iterate, we have estimates of higher order derivatives and obtain
$$
\widetilde g(x)=O_m(|x|^{-\zeta})+\left\{
\begin{array}{llll}
O_{m-1}( | x | ^ {-2\zeta}), & \text{if }0<\zeta<1,\\
O_{m-1}(|x|^{-2}(\ln|x|)^2)& \text{if }\zeta=1\text{ and }n\geq4,\\
O_{m-1}(|x|^{-2}(\ln|x|)^4), & \text{if }\zeta=1\text{ and }n=3.\\
\end{array}
\right.
$$
By Lemmas \ref{Lem-existence-fastdecay} and \ref{lem-existence-fastd-deri}, there exists a solution $v_{\widetilde g}$ of \eqref{temp-1} on $\mathbb R^n\setminus \overline{B_1}$ with
$$
\widetilde v_{\widetilde g}(x)=\left\{
\begin{array}{llll}
  O_{m}(|x|^{2-\zeta}), & \text{if }0<\zeta<1,\\
  O_{m}(|x|(\ln|x|)), & \text{if }\zeta=1.\\
\end{array}
\right.
$$
Since $\widetilde v-\widetilde v_{\widetilde g}$ is harmonic on  $\mathbb R^n\setminus\overline{B_1}$ and $\widetilde v-\widetilde v_{\widetilde g}=o(|x|^2)$ as $|x|\rightarrow\infty$, by spherical harmonic decomposition we have
$$
\widetilde v-\widetilde v_g=O_l(|x|)\quad\text{as }|x|\rightarrow+\infty
$$
for any $l\in\mathbb N$. By rotating back and applying interior Schauder estimates again, we have the desired results.

If $1<\zeta\leq 2$, then \eqref{equ-VaniRate-Hessian} implies   $|D^2v|=O(|x|^{-1})$ as $|x|\rightarrow\infty$.
Since  $D_ef=O_{m-1}(|x|^{-\zeta-1})$ at infinity and the coefficients of Eq.~\eqref{linearized-equation-1} has uniformly bounded $C^{\alpha}$ norm,
by Theorem \ref{thm:linearLimit} there
exists $ b _e\in\mathbb R$ such that

\begin{equation}\label{capture1-order}  D_ { e }v ( x ) =  b _ { e } +
\left\{
\begin{array}{llll}
O \left( | x | ^ { 2 - \min\{n,\zeta+1\}} \right), & \zeta\not=n-1,\\
O \left( | x | ^ { 2 - n}(\ln|x|) \right), & \zeta=n-1,\\
\end{array}
\right.
\quad\text{as } | x | \rightarrow \infty.
\end{equation}
Picking $e$ as $n$ unit coordinate vectors of $\mathbb R^n$, we found $ b \in\mathbb R^n$ from \eqref{capture1-order} and let

$$w( x ): = v ( x ) -  b \cdot  x = u ( x ) - \left( \frac { 1 } { 2 } x ^ { T } A x +  b \cdot  x \right).$$
By (\ref{capture1-order}), since $n-1\geq 2$ and $1<\zeta\leq 2$,
\begin{equation}\label{equ-VaniRate-Gradient}
|Dw(x)|=|(D_{x_1}v- b _1,\cdots,D_{x_n}v- b _n)|=
\left\{
\begin{array}{lllll}
O(|x|^{1-\zeta}), & \text{if }1<\zeta<2,\\
& \text{or }\zeta=2\text{ and }n>3,\\
O(|x|^{-1}(\ln|x|)), & \text{if }\zeta=2~\text{and }n=3,\\
\end{array}
\right.
\end{equation}
as $|x|\rightarrow\infty$.
By a direct computation, \eqref{equ-VaniRate-Gradient} yields
\[
|w(x)|=\left\{
\begin{array}{lllll}
O_1(|x|^{2-\zeta}), & \text{if }1<\zeta<2,\\
O_1(\ln|x|),& \text{if }\zeta=2\text{ and }n>3,\\
O_1((\ln|x|)^2), & \text{if }\zeta=2~\text{and }n=3,\\
\end{array}
\right.
\]
as $|x|\rightarrow\infty$. Similar to the proof of Lemma \ref{lem:initial-n>3} we set
\[
w_{R}(y):=\left(\frac{4}{R}\right)^{2} w\left(x+\frac{R}{4} y\right), \quad|y| \leq 2.
\]
Then
\[D^2w_R(y)=D^2w(x+\frac{R}{4}y)=D^2v_R(y)\quad\text{and}\quad
F(A+D^2w_R(y))=f_R(y)\quad\text{in }B_2.
\]
For any $e\in\partial B_1$, by taking partial derivative to the equation above, we have
\[
\hat{a_{ij}}(y)D_{ije}w_R(y)=D_ef_R(y)\quad\text{in }B_2,
\]
where the coefficients are uniformly (to $R$) elliptic with uniform $C^{\alpha}$-norm in $B_1$. By interior Schauder estimate and taking further derivatives, there exists $C>0$ independent of $R$ such that
\[
|D^{k}w_R(0)|\leq \left\{
\begin{array}{llll}
  CR^{-\zeta}, & \text{if }1<\zeta<2,\\
  CR^{-2}\ln R, & \text{if }\zeta=2\text{ and }n> 3,\\
  CR^{-2}(\ln R)^2, & \text{if }\zeta=2\text{ and }n=3,\\
\end{array}
\right.
\]
for all $k=0,1,\cdots,m+1$. Similar to the previous case, we set $Q$ as in \eqref{equ:defni-Q} and $\tilde v(x):=w(Qx)$. Then by the computation above we have
$$
\widetilde g(x)=O_m(|x|^{-\zeta})+\left\{
\begin{array}{llll}
O_{m-1}(|x|^{-2\zeta}), & \text{if }1<\zeta<2,\\
O_{m-1}(|x|^{-4}(\ln|x|)^2)  &\text{if }\zeta=2\text{ and }n>3,\\
  O_{m-1}(|x|^{-4}(\ln|x|)^4), & \text{if }\zeta=2\text{ and }n=3.\\
\end{array}
\right.
$$
By Lemmas \ref{Lem-existence-fastdecay} and \ref{lem-existence-fastd-deri}, there exists a solution $v_{\widetilde g}$ of \eqref{temp-1} on $\mathbb R^n\setminus \overline{B_1}$ with
$$
\widetilde v_{\widetilde g}(x)=\left\{
\begin{array}{llll}
  O_{m}(|x|^{2-\zeta}), & \text{if }1<\zeta<2,\\
  O_{m}((\ln|x|)), & \text{if }\zeta=2.\\
\end{array}
\right.
$$
Since $\widetilde v-\widetilde v_{\widetilde g}$ is harmonic on  $\mathbb R^n\setminus\overline{B_1}$ and $\widetilde v-\widetilde v_{\widetilde g}=o(|x|)$ as $|x|\rightarrow\infty$, by spherical harmonic decomposition we have
$$
\widetilde v-\widetilde v_g=O_l(1)\quad\text{as }|x|\rightarrow+\infty
$$
for any $l\in\mathbb N$.
By rotating back and applying interior Schauder estimates again, we obtain the results in  $1<\zeta\leq 2$ cases in \eqref{equ-beha-n>3}.
\end{proof}

\section{Proof for $n=2$ case}\label{seclabel-n=2}

In this section, we prove Theorem \ref{thm:main-n=2}. In $n=2$ case, since  Theorem \ref{thm:linearLimit} may fails,  we apply the iterate method as in \cite{Bao-Li-Zhang-ExteriorBerns-MA,Liu-Bao-2021-Dim2-MeanCur} etc. For reading simplicity, we introduce the following results.

\begin{lemma}\label{lem:initial-n=2}
  Let $u,f$ be as in Theorem \ref{thm:main-n=2}, $A,\epsilon$ be as in  \eqref{equ-roughConverge}  and $w:=u-\frac{1}{2}x^TAx$. Then there exist $C,\alpha$ and $\epsilon'>0$ such that
    \begin{equation}\label{Chap6-Asy-Initial}
|D^kw(x)|\leq C|x|^{2-k-\epsilon'}\quad\text{and}\quad
\dfrac{|D^{m+1}w(x_1)-D^{m+1}(x_2)|}{|x_1-x_2|^{\alpha}}\leq C|x_1|^{1-m-\epsilon'-\alpha}
\end{equation}
for all  $ | x | > 2, k = 0 , \cdots , m + 1$ and $\left| x_ { 1 } \right| > 2,x _ { 2 } \in B _ { \left| x _ { 1} \right| / 2 } \left( x _ { 1 } \right)$.

Furthermore, we have an iterative structure that if \eqref{Chap6-Asy-Initial} holds for some $0<\epsilon'<\min\{\frac{\zeta}{2},\frac{1}{2}\}$, then it holds also for $\epsilon'$ replaced by $2\epsilon'$ with another constant $C$.
\end{lemma}
The proof of \eqref{Chap6-Asy-Initial} is omitted here since it is similar to Lemma 2.1 in \cite{Bao-Li-Zhang-ExteriorBerns-MA} or Lemma 4.1 in \cite{Liu-Bao-2021-Dim2-MeanCur}, which is based only on the interior estimates by Caffarelli \cite{Caffarelli-InteriorEstimates-MA} and Jian--Wang \cite{Jian-Wang-ContinuityEstimate-MA} and interior Schauder estimates.
The proof of iterative structure can be found as  Lemma 2.2 in \cite{Bao-Li-Zhang-ExteriorBerns-MA}, which relies on the assumption that $m\geq 3$ and is different from higher dimension case.

Now we are ready to prove Theorem \ref{thm:main-n=2} by the iterative structure above.
\begin{proof}[Proof of Theorem \ref{thm:main-n=2}]
  By Lemma \ref{lem:initial-n=2}, there exist $\alpha$ and $\epsilon'>0$ such that \eqref{Chap6-Asy-Initial} holds.

  If $0<\zeta\leq 1$, we let $p_1\in\mathbb N$ be the positive integer such that
\[
  2^{p_1}\epsilon'<\zeta\quad\text{and}\quad \zeta<2^{p_1+1}\epsilon'<2\zeta.
\]
  (If necessary, we may choose $\epsilon'$ smaller to make both inequalities hold.) Let $\epsilon_1:=2^{p_1}\epsilon'$. Applying the iterative structure in Lemma \ref{lem:initial-n=2} $p_1$ times, we have
  \begin{equation}\label{Chap6-equ-converge5}
  |D^kw(x)|\leq C|x|^{2-k-\epsilon_1}\quad\text{and}\quad
\dfrac{|D^{m+1}w(x_1)-D^{m+1}(x_2)|}{|x_1-x_2|^{\alpha}}\leq C|x_1|^{1-m-\epsilon_1-\alpha}
  \end{equation}
  for all  $ | x | > 2, k = 0 , \cdots , m + 1$ and $\left| x_ { 1 } \right| > 2,x _ { 2 } \in B _ { \left| x _ { 1} \right| / 2 } \left( x _ { 1 } \right)$.

  Applying Newton--Leibintz formula between Eq.~\eqref{equ:MA} and $F(A)=f(\infty)$, we have
  $$
  D_{M_{ij}}F(A)D_{ij}w=f(x)-f(\infty)+\left(\widetilde{a_{ij}}(x)-\widetilde{a_{ij}}(\infty)\right)D_{ij}w=:g_1(x)
  $$
  in $\mathbb R^2$, where $w$ is defined as in Lemma \ref{lem:initial-n=2},   the coefficients are uniformly elliptic and
  $$
  \widetilde{a_{ij}}(x)=\int_0^1D_{M_{ij}}F(A+tD^2w(x)){\mathrm{d}}t=D_{M_{ij}}F(A)+O_{m-1}(|x|^{-\epsilon_1})
  $$
  as $|x|\rightarrow\infty$. Let $Q:=[D_{M_{ij}}F(A)]^{\frac{1}{2}}$ and $\widetilde w(x):=w(Qx)$. By the invariance of trace under cyclic permutations again, we have
  \begin{equation}\label{equ-temp-5}
  \Delta\widetilde w=\widetilde g_1(x):=g_1(Qx).
  \end{equation}
  By the definition of $\widetilde a_{ij}(x)$, condition \eqref{equ:cond-f} on $f$ and \eqref{Chap6-equ-converge5} we have
  $$
  \widetilde g_1(x)=O_{m}(|x|^{-\zeta})+O_{m-1}(|x|^{-2\epsilon_1})=O_{m-1}(|x|^{-\zeta})
  $$
  as $|x|\rightarrow\infty$. By Lemmas \ref{Lem-existence-fastdecay} and \ref{lem-existence-fastd-deri}, there exists a function $\widetilde w_{\widetilde g_1}$ solving \eqref{equ-temp-5} on $\mathbb R^2\setminus\overline{B_1}$ such that
  $$
  \widetilde w_{\widetilde g_1}(x)=
  \left\{
  \begin{array}{lllll}
  O_m(|x|^{2-\zeta}), & \text{if }0<\zeta<1,\\
  O_m(|x|^{2-\zeta}(\ln|x|)), & \text{if }\zeta=1,\\
  \end{array}
  \right.
  $$
  as $|x|\rightarrow\infty$. Since $\widetilde g_1-\widetilde w_{\widetilde g_1}$ is harmonic on  $\mathbb R^2\setminus\overline{B_1}$ and $\widetilde g_1-\widetilde w_{\widetilde g_1}=o(|x|^2)$ at infinity, by spherical harmonic decomposition we have
  $$
  \widetilde g_1-\widetilde w_{\widetilde g_1}=O_l(|x|)\quad\text{as }|x|\rightarrow\infty
  $$
  for any $l\in\mathbb N$.
  By rotating back and applying  interior Schauder estimates as in the proof of Theorem \ref{thm:main-n>3}, we finish the proof of $0<\zeta\leq 1$ cases in \eqref{equ-beha-n=2}.

  If $1<\zeta\leq 2$, we let $p_2\in\mathbb N$ be the positive integer such that $$
  2^{p_2}\epsilon'<1\quad\text{and}\quad 1<2^{p_2+1}\epsilon'<2.
  $$
  (If necessary, we may choose $\epsilon'$ smaller to make both inequalities hold.)
  Let $\epsilon_2:=2^{p_2}\epsilon'$.
  Applying the iterative structure in Lemma \ref{lem:initial-n=2} $p_2$ times, we have \eqref{Chap6-equ-converge5} with $\epsilon_1$ replaced by $\epsilon_2$ for all  $ | x | > 2, k = 0 , \cdots , m + 1$ and $\left| x_ { 1 } \right| > 2,x _ { 2 } \in B _ { \left| x _ { 1} \right| / 2 } \left( x _ { 1 } \right)$.

  Similar to the strategy above we apply Newton--Leibnitz formula and rotation $Q:=[D_{M_{ij}}F(A)]^{\frac{1}{2}}$ to obtain \eqref{equ-temp-5}. By the definition of $\widetilde a_{ij}(x)$, condition \eqref{equ:cond-f} on $f$ and \eqref{Chap6-equ-converge5} we have
  $$
  \widetilde g_1(x)=O_{m}(|x|^{-\zeta})+O_{m-1}(|x|^{-2\epsilon_2})=O_{m-1}(|x|^{-\min\{\zeta,2\epsilon_2\}})
  $$
  as $|x|\rightarrow\infty$. By Lemmas \ref{Lem-existence-fastdecay}  and \ref{lem-existence-fastd-deri}, there exists a function $\widetilde w_{\widetilde g_1}$ solving \eqref{equ-temp-5} on $\mathbb R^2\setminus\overline{B_1}$ such that
  $$
  \widetilde w_{\widetilde g_1}(x)=O_m(|x|^{2-\min\{\zeta,2\epsilon_2\}})
  $$
  as $|x|\rightarrow\infty$. Since $\widetilde w-\widetilde w_{\widetilde g_1}$ is harmonic on  $\mathbb R^2\setminus\overline{B_1}$ and $\widetilde w-\widetilde w_{\widetilde g_1}=o(|x|^2)$ at infinity, by spherical harmonic decomposition we have $\widetilde b \in\mathbb R^2$ such that
  $$
  \widetilde w-\widetilde w_{\widetilde g_1}=\widetilde b \cdot x+O_l(\ln|x|)\quad\text{as }|x|\rightarrow\infty
  $$
  for any $l\in\mathbb N$. Consequently, by setting
  $$
  \widetilde w_1(x):=\widetilde w(x)-\widetilde b \cdot x,
  $$
  we have
  $$
  \begin{array}{lllll}
    \widetilde w_1(x) &= & O_l(\ln|x|)+O_m(|x|^{2-\min\{\zeta,2\epsilon_2\}})\\
    &=& O_m(|x|^{2-\min\{\zeta,2\epsilon_2\}})\\
  \end{array}
  $$
  as $|x|\rightarrow\infty$. From the proof of \eqref{temp-1}, by a direct computation and interior estimates we have
  $$
  \Delta\widetilde w_1=\widetilde g_2(x)=O_m(|x|^{-\zeta})+O_{m-1}(|x|^{-2\min\{\zeta,2\epsilon_2\}})
  $$
  in $\mathbb R^2\setminus\overline{B_1}$. Since $$
  1<\zeta\leq 2<2\min\{\zeta,2\epsilon_2\},
  $$ by Lemmas  \ref{Lem-existence-fastdecay} and \ref{lem-existence-fastd-deri}, there exists a function $\widetilde w_{\widetilde g_2}$ solving \eqref{equ-temp-5} on $\mathbb R^2\setminus\overline{B_1}$ such that
  $$
  \widetilde w_{\widetilde g_2}(x)=
  \left\{
  \begin{array}{lll}
  O_m(|x|^{2-\zeta}), & \text{if }1<\zeta<2,\\
  O_m( (\ln|x|)^2), & \text{if }\zeta=2,\\
  \end{array}
  \right.
  $$
  as $|x|\rightarrow\infty$. Since $\widetilde w_1-\widetilde w_{g_2}$ is harmonic on $\mathbb R^2\setminus\overline{B_1}$ and $\widetilde w_1-\widetilde w_{g_2}=o(|x|)$ at infinity, by spherical harmonic decomposition we have
  $$
  \widetilde w_1-\widetilde w_{g_2}=O_l(\ln|x|)\quad\text{as }|x|\rightarrow\infty
  $$
  for all $l\in\mathbb N$.
  By rotating back and applying interior Schauder estimates as in the proof of Theorem \ref{thm:main-n>3}, we finish the proof of  \eqref{equ-beha-n=2}.
\end{proof}

\small

\bibliographystyle{plain}

\bibliography{C:/Bib/Thesis}

\bigskip

\noindent Z.Liu \& J. Bao

\medskip

\noindent  School of Mathematical Sciences, Beijing Normal University\\
Laboratory of Mathematics and Complex Systems, Ministry of Education\\
Beijing 100875, China \\[1mm]
Email: \textsf{liuzixiao@mail.bnu.edu.cn, jgbao@bnu.edu.cn}

\end{document}